\documentclass[12pt,a4paper]{amsart}
\usepackage{amsfonts}
\usepackage{amssymb}
\usepackage{ifthen}
\usepackage{graphicx}
\usepackage{theoremref}
\nonstopmode \numberwithin{equation}{section}
\usepackage[usenames]{color}
\newtheorem{thm}{Theorem}
\newtheorem{lem}{Lemma}
\newtheorem{cor}{Corollary}
\newtheorem{ques}{Question}

\newtheorem{cl}{Claim}
\newtheorem{ca}{Case}
\newtheorem{sca}{Subcase}
\newtheorem{scl}{Subclaim}
\newtheorem{conj}[equation]{Conjecture}

\theoremstyle{definition}
\newtheorem{defn}{Definition}
\newtheorem{example}{Example}[section]
\newtheorem{op}[equation]{Open Problem}
\newtheorem{rem}{Remark}
\newcounter {own}
\def\theown {\thesection       .\arabic{own}}

\newenvironment{pf}[1][]{%
 \vskip 3mm
 \noindent
 \ifthenelse{\equal{#1}{}}%
  {{\slshape Proof. }}%
  {{\slshape #1.} }%
 }%
{\qed\bigskip}

\DeclareMathOperator*{\esslim}{ess\,lim}

\newcounter{alphabet}
\newcounter{tmp}
\newenvironment{Thm}[1][]{\refstepcounter{alphabet}%
\bigskip%
\noindent%
{\bf Theorem \Alph{alphabet}}%
\ifthenelse{\equal{#1}{}}{}{ (#1)}%
{\bf .} \itshape}{\vskip 8pt}

\makeatletter
\renewcommand{\Ref}[1]{\@ifundefined{r@#1}{}{\setcounter{tmp}{\ref{#1}}\Alph{tmp}}}
\makeatother

\def\be{\begin{equation}}
\def\ee{\end{equation}}

\newcommand{\bee}{\begin{enumerate}}
\newcommand{\eee}{\end{enumerate}}

\newcommand{\blem}{\begin{lem}}
\newcommand{\elem}{\end{lem}}
\newcommand{\bthm}{\begin{thm}}
\newcommand{\ethm}{\end{thm}}
\newcommand{\bcor}{\begin{cor}}
\newcommand{\ecor}{\end{cor}}
\newcommand{\beg}{\begin{exam}}
\newcommand{\eeg}{\end{exam}}
\newcommand{\begs}{\begin{examples}}
\newcommand{\eegs}{\end{examples}}
\newcommand{\bdefe}{\begin{defn}}
\newcommand{\edefe}{\end{defn}}
\newcommand{\bprob}{\begin{prob}}
\newcommand{\eprob}{\end{prob}}
\newcommand{\bques}{\begin{ques}}
\newcommand{\eques}{\end{ques}}
\newcommand{\bei}{\begin{itemize}}
\newcommand{\eei}{\end{itemize}}
\newcommand{\bcon}{\begin{conj}}
\newcommand{\econ}{\end{conj}}
\newcommand{\bop}{\begin{op}}
\newcommand{\eop}{\end{op}}

\newcommand{\bca}{\begin{ca}}
\newcommand{\eca}{\end{ca}}
\newcommand{\bsca}{\begin{sca}}
\newcommand{\esca}{\end{sca}}

\newcommand{\bcl}{\begin{cl}}
\newcommand{\ecl}{\end{cl}}

\newcommand{\bscl}{\begin{scl}}
\newcommand{\escl}{\end{scl}}

\newcommand{\bcons}{\begin{conjs}}
\newcommand{\econs}{\end{conjs}}
\newcommand{\bprop}{\begin{propo}}
\newcommand{\eprop}{\end{propo}}
\newcommand{\br}{\begin{rem}}
\newcommand{\er}{\end{rem}}
\newcommand{\brs}{\begin{rems}}
\newcommand{\ers}{\end{rems}}
\newcommand{\bo}{\begin{obser}}
\newcommand{\eo}{\end{obser}}
\newcommand{\bos}{\begin{obsers}}
\newcommand{\eos}{\end{obsers}}
\newcommand{\bpf}{\begin{pf}}
\newcommand{\epf}{\end{pf}}
\newcommand{\ba}{\begin{array}}
\newcommand{\ea}{\end{array}}
\newcommand{\beq}{\begin{eqnarray}}
\newcommand{\beqq}{\begin{eqnarray*}}
\newcommand{\eeq}{\end{eqnarray}}
\newcommand{\eeqq}{\end{eqnarray*}}

\newcounter{minutes}
\divide\time by 60
\newcounter{hours}
\multiply\time by 60 \addtocounter{minutes}{-\time}
\begin{document}
\bibliographystyle{amsplain}
\title [] {A note on polyharmonic mappings}
\author{D. Bshouty, S. Evdoridis, and A. Rasila$^*$}

\begin{abstract}
In this paper we prove a Rad\'o type result showing that there is no univalent polyharmonic mapping of the unit disk onto the whole complex plane. We also establish a connection between the boundary functions of harmonic and biharmonic mappings. Finally, we show how a close-to-convex biharmonic mapping can be constructed from a convex harmonic mapping.
\end{abstract}

\thanks{The research was supported by Academy of Finland (No. 308063), NNSF of China (No. 11971124), NSF of the Guangdong Province, and  Foundation for Aalto University Science and Technology.\\$^*$ Corresponding author}

\maketitle
\section{Introduction and Preliminaries}

A complex-valued function $F\colon D \to \mathbb{C}$, with $F(z)=u(z)+iv(z)$, is called a {\it harmonic mapping} in a plane domain $D\subset \mathbb{C}$, if both $u,v\colon \mathbb{R}^2 \to \mathbb{R}$ are harmonic functions in $D$. Then, we write $\Delta F=0$, where $\Delta $ is the complex Laplacian operator
$$\Delta = 4 \frac{\partial ^2}{\partial z \partial \overline{z}} = \frac{\partial ^2}{\partial x^2} + \frac{\partial ^2}{\partial y^2}.$$
For a simply connected domain $D$, every harmonic mapping $F\colon D\to \mathbb{C}$ in $D$ can be expressed in the form $F=h+ \overline{g}$, where $h$ and $g$ are analytic functions in $D$. The mapping $F$ is locally univalent in $D$ if, and only if, it has a non-vanishing Jacobian, $$J_F(z):=|h'(z)|^2-|g'(z)|^2 \neq 0,$$
in $D$. If $J_F(z)>0$ ($J_F(z)<0$), $F$ is { \it sense-preserving} ({\it sense-reversing}) in $D$. More details about harmonic mappings can be found in \cite{CS, D, PR}.

\subsection{Biharmonic and Polyharmonic Mappings}
A function $f$, defined in a domain $D\subset \mathbb{C}$, is called a {\it polyharmonic}  (or {\it $p$-harmonic}) mapping, if  $\Delta ^p f =0$ for $p\in \mathbb N\geq 1$, where  $\Delta ^p f =\Delta (\Delta ^{p-1} f)$ and $\Delta ^1 f = \Delta f$. Clearly, for $p=1$, $f$ is a harmonic mapping.  When $D$ is a simply connected domain $f$ admits the representation
\begin{equation}
\label{form}
 f(z)=\sum _{k=1}^p |z|^{2(k-1)} G_k(z),
\end{equation}
where $G_k$; $k=1,2,\ldots ,p$ are harmonic mappings in $D$. In the notable special case where $p=2$, the function satisfies the equation $\Delta^2 f=0$ and it is known as a {\it biharmonic mapping}. These functions are related with several applications in fields such as engineering, where connections with elasticity theory and fluid dynamics can be found in the literature (see {\it e.g.}  \cite{HB, Kh, L}).

By using the representation \eqref{form}, many properties of harmonic mappings can be generalized to the polyharmonic mappings. This line of research was started in by S. Chen, Ponnusamy, Qiao, and Wang in \cite {CPW, QW}, and continued by the third author with J. Chen and Wang in  the series of articles  {\cite {CRW1, CRW2, CRW3, CRW4, CRW5} where, for example, the radii of univalency, starlikeness and convexity were studied for polyharmonic mappings.

On the other hand, many popular techniques used in studying harmonic mappings are not available in the biharmonic case, and new methodologies are required. This paper is an attempt to this direction, with the key ingredient being the use of auxiliary harmonic mappings of the type
\begin{equation}
\label{auxharm}
 g_r(z) = \sum_{k=1}^p r^{2(k-1)} F_k(z),
\end{equation}
where $r\in (0,1)$. Besides its significance from the point of view of applications, this case is generally more pleasant to study as it leads to easily traceable formulas. For this reason, the last three of our main results are only formulated in the biharmonic case.

\subsection{Close-to-convex domains and mappings}
A domain $D$ is said to be {\it close-to-convex} if its complement can be expressed as a union of closed half-lines such that the corresponding open half-lines are disjoint. We now restrict our attention on functions defined in the unit disk $\mathbb{D}$. We denote by $S_H$ the class of univalent sense-preserving harmonic mappings, defined in $\mathbb{D}$, with $F(0)=0$ and $F_z(0)=1$. The subclass of $S_H$, for which $F(\mathbb{D})$ is convex, is denoted by $K_H$.

An analytic function $h(z)$ is said to be close-to-convex in $\mathbb D$ if $h(\mathbb D)$ is a
close-to-convex domain. Every close-to-convex analytic function is univalent. Close-to-convexity is a {\it hereditary property} in the sense that for $\mathbb D_r=\{|z|<r\};\;\; 0<r<1$ the image set $h(\mathbb D_r)$ is a close-to-convex domain.

A general complex valued mapping $f$ in $\mathbb{D}$ such that $f(\mathbb D)$ is a close-to-convex domain, may neither be univalent, nor satisfy the above hereditary close-to-convexity property. In this paper, a mapping $f$ of the unit disk onto a close-to-convex domain that is univalent is termed a close-to-convex mapping. If, in addition, $f$ satisfies the hereditary property, then it is termed {\it fully close-to-convex}. These concepts are applicable both to harmonic and polyharmonic mappings.


\section{Main Results}

In this paper, we discuss the following three types of results for polyharmonic and biharmonic mappings.

\subsection*{{\bf I.} Rad\'o's Theorem for Polyharmonic Mappings}  Tibor Rad\'o \cite{R} was the first who stated and proved a special case of the so-called Rad\' o's Theorem for harmonic mappings. Proofs of more general forms of the theorem were given later by Bers \cite { B} and  Nitsche {\cite N}. The statement of the theorem is the following.

\begin{Thm} [{\rm Rad\' o}] \label{Rado}
There is no univalent harmonic mapping from $\mathbb{D} $ onto $\mathbb{C}$.
\end{Thm}

Our goal is to prove the same theorem for polyharmonic mappings. We shall need
 the well-known Kneser theorem. In 1926, Rad\'o posed the problem, if any homeomorphism of the unit circle onto the boundary of a bounded convex domain, extends to a univalent harmonic mapping on the whole disk. Kneser, in the same volume \cite{Kn}, proved the following more general theorem.

\begin{Thm} [{\rm Kneser}] \label{Kneser}
Let $\Omega$ be a bounded simply connected Jordan domain. Consider a homeomorphism $g^*(e^{it})$ from $\partial \mathbb D$  onto $\partial\Omega.$ Let $g(z)$ be the harmonic extension of $g^*$ into the disk. If $g(\mathbb D)\subset \Omega$ then $g$ is univalent in $\mathbb D.$
\end{Thm}

In 1945, Choquet gave a more analytic solution to Rad\'o's problem in \cite{Cho}. Nowadays, this theorem is commonly known as the Rad\'o-Kneser-Choquet theorem.

In this paper Theorem \Ref{Rado} and Theorem \Ref{Kneser}  will be used to characterize univalent polyharmonic mappings in $\mathbb D$ via  terms of univalent harmonic mappings to prove the following version of Rad\'o's theorem for polyharmonic mappings. 

\begin{thm}
\label{main1}
There is no univalent polyharmonic mapping of the unit disk onto the whole complex plane.
\end{thm}

\subsection*{{\bf II.} Boundary Behaviour of Biharmonic Mappings}
There are only few results concerning the boundary behaviour of polyharmonic mappings. As it is difficult to study the harmonic case, the polyharmonic case appears to be even more complicated. On the other hand, auxiliary harmonic mappings of the form \eqref{auxharm} coincide with the polyharmonic mapping of the form \eqref{form} on the circles $|z|=r<1$, which suggests that such mappings could be used in studying the boundary behavior of polyharmonic mappings.

Indeed, based on the fact that in some cases the boundary function  of a biharmonic (or more generally polyharmonic) mapping in $\mathbb D$ coincides with the boundary of an associated harmonic mapping there, we show that one can compare the behaviours of these functions to each other.  For this purpose, we define the boundary function $f^*$ of a univalent mapping $f$ to be the unrestricted limit
$$f^*(e^{it})=\lim_{z\rightarrow e^{it}}f(z),$$
if it exists.

 The classical Riemann Mapping Theorem concerns the existence of a normalized univalent analytic  function between two simply connected domains, say $\mathbb D$ and $\Omega$. If $\Omega$ is a bounded Jordan domain with a rectifiable boundary, the function extends to a univalent continuous function from $\overline{\mathbb D}$ onto $\overline{\Omega}.$ In the harmonic case, we seek a normalized harmonic univalent function $f(z)=h(z)+\overline{g(z)}$  with a specific dilatation $a(z)= g'(z)/h'(z)$ that is bounded by one.
 The existence of such function is guaranteed if $|a(z)|<k<1.$ However, for $|a(z)|<1,$ such a function``onto" $\Omega$ may not exist, and the best we can do is to talk about a generalized Riemann Mapping Theorem (GRM) that asks for  a univalent harmonic mapping with dilatation $a(z)$ bounded by one ``onto" $\Omega$ in the sense that $f(\partial\mathbb D)$ is mapped into $\partial \Omega$ and the mapping between those be sense-preserving and  one to one.

 A weak version of the GRM theorem for harmonic functions (\cite[Theorem 4.3]{HS} and  \cite[Theorem A]{BLW}) states that for a univalent harmonic mapping with dilatation $|a(z)|<1$ from $\mathbb{D}$ onto a (bounded)  rectifiable Jordan domain $\Omega$, its boundary function is continuous except for a countable number of jumps. A  jump at $e^{it}\in \partial \mathbb{D}$ happens whenever $f^*(e^{it-})= \esslim_{s\uparrow t} f^*(e^{is})$ and $f^*(e^{it+})= \esslim_{s\downarrow t} f^*(e^{is})$ exist, are finite and are not equal. In this case the cluster set of the harmonic function at $e^{it}$ is a straight line segment between
 $f^*(e^{it-})$ and $f^*(e^{it+}).$

The next result of Bshouty, Lyzzaik, and Weitsman \cite{BLW} about the behaviour of the boundary function of a harmonic mapping, is  useful for examining a more general setting, such as the biharmonic case.  One version of this result is the following.

\begin{Thm}\label{B-L-W} Let $f=h+\overline{g}$ be a univalent harmonic mapping from $\mathbb{D}$ into a bounded  rectifiable Jordan domain $\Omega.$ Then, the boundary function $f^*$ of $f$ is discontinuous at the point $e^{i\theta _0} \in \partial \mathbb{D}$ if, and only if
$$\lim_{r\to 1^-}(1-r)h'(re^{i\theta _0})=c,\;\; c\not =0.$$
\end{Thm}

In the following we show a result about the boundary function of a univalent polyharmonic mapping, defined in the unit disk. The next theorem holds true for univalent polyharmonic mappings though it is written for biharmonic univalent functions for simplicity.

\begin{thm} \label{main2}
 Let $f(z)=F_1(z)+|z|^2F_2(z)$ be a univalent biharmonic mapping of $\mathbb{D}$ onto a bounded and rectifiable Jordan domain $\Omega,$  where $F_1=H_1+\overline{G_1}$ and $F_2=H_2+\overline{G_2}$ are two harmonic mappings in $\mathbb{D}$. We assume that $F_2(z)=o\left(\frac 1{1-z}\right )$. Let $f^*(e^{it})$ denote the boundary function of $f(z).$ Then $f^*$ has a jump at $e^{i\theta_0}$ if, and only if
$$\lim_{r\to 1^{-}}(1-r)|H_1'(re^{i\theta_0})+H_2'(re^{i\theta_0})|=c>0.$$
\end{thm}

Next, we define a half-closed curve, contained in the unit disk. Denote by $\Gamma _{m,\theta _0}$ the curve, which is completely contained in $\mathbb{D}$, given by $$\Gamma_{m, \theta_0 }(\theta)=(1 -m|\theta -\theta _0|)e^{i\theta},$$ $e^{i\theta_0}\in \partial \mathbb{D}$, $m>0$ and $0<|\theta - \theta _0| \leq \min\{\pi, 1/m\}$. 

\begin{thm} \label{main4}
Let $\Omega$ be a bounded, convex, simply-connected domain and $f(z)=F_1(z)+|z|^2F_2(z)$ be a univalent biharmonic mapping of $\mathbb{D}$ onto $\Omega$. If 
$F_1=H_1+\overline{G_1}$ and $F_2=H_2+\overline{G_2}$ are two harmonic mappings in $\mathbb{D}$ with $F_2(z)=o\left(\frac 1{1-z}\right )$ and
$$| \phi(z)|:= \left| \frac{G'_1(z) + G'_2(z)}{H'_1(z)+H'_2(z)} \right| <1,$$
then
$$\int_{\Gamma _{m,\theta_0}} \frac{1-| \phi(z)|^2}{1-|z|^2}|dz|=\infty, \text{  for all  } 0<m<1/\pi,$$
implies that the boundary function of $f$ is continuous at $e^{i\theta_0}$.
\end{thm}

\subsection*{{\bf III.} Close-to-Convex Biharmonic Mappings}
In the last part of this paper, we are focused on univalent close-to-convex biharmonic mappings. We show how such a mapping can be constructed by using a member of the class $K_H$ under certain conditions.

Lemma \ref{lem1} formulated in Section \ref{mainproofs} can be used to construct a univalent biharmonic mapping $f=F_1+|z|^2 F_2$ from a univalent harmonic mapping $g=F_1+F_2$ restricted to two conditions:
\begin{itemize}
\item[(i)]
For every $0<r<1,$ the harmonic function $g_r(z)$ is univalent, and
\item[(ii)]
\[
|a_f|= \Big|\frac{f_{\overline z}}{f_z}\Big|<1.
\]
\end{itemize}
We introduce a method to produce univalent fully close-to-convex biharmonic mappings from convex harmonic mappings, with merely
condition (ii).

\begin{thm}\label{main3}
Let $H$ and $G$ be analytic functions and $F=H+\overline{G}\in K_H$ be a convex harmonic mapping in $\mathbb{D}$. If  $f(z)=H(z)+|z|^2 \overline{G(z)}$  is locally univalent, then $f$ is a univalent fully close-to-convex biharmonic mapping in $\mathbb{D}$.
\end{thm}

\section{Proofs of the Main Results}\label{mainproofs}

In this section, we give the proofs of the main results. We start with the following lemma, which is of independent interest.

\begin{lem}\label{lem1}
A locally univalent sense-preserving polyharmonic mapping on the unit disk $\mathbb D,$
$$f(z)= \sum_{k=1}^p|z|^{2(k-1)} F_k(z),$$
where $F_j, j=1,\ldots,p$ are harmonic mappings, is one-to-one, if, and only if, for every $r,\; 0<r<1$ the harmonic function
$$ g_r(z) = \sum_{k=1}^p r^{2(k-1)} F_k(z),$$
is univalent harmonic mapping in $\mathbb D_r.$
\end{lem}
\begin{proof}
Assume that $f(z)$ is univalent in $\mathbb D$ and set $f_r(z)=f(z)|_{\mathbb D_r}.$ One immediately observes that for $r,\;0<r<1,$
$f_r(\mathbb D_r)$ is a chain of outward open domains $\Omega_r,$ whose boundary functions $f^*_r$  are univalent and disconnected, {\it i.e.} for $0<r<R<1,$
\begin{equation}\label{lem1.1}
  f^*_r(\partial\mathbb D_r)\cap f^*_R(\partial\mathbb D_R)=\emptyset.
\end{equation}
The more so
\begin{equation}\label{lem1.2}
f_r(\partial \mathbb D_r)\equiv g_r(\partial \mathbb D_r).
\end{equation}

Assume to the contrary that for some  $ 0<r<1, \;g_r(z)$ is not univalent in $\mathbb D_r.$ Since $g_r$ is a harmonic mapping univalent on $\partial\mathbb D_r,$ by Theorem B, at some interior point,
$\zeta\in \mathbb D_r,\;\; g_r(\zeta)\not\in f_r(\mathbb D_r).$  As such, there exists $\rho>r$ such that $g_r(\zeta)\in f_\rho(\partial\mathbb D_\rho).$ On the other hand, $|\zeta|=\varrho< r,$ and by (\ref{lem1.2}) $g_r(\zeta)\in f_\varrho(\partial\mathbb D_\varrho).$ This contradicts (\ref{lem1.1}).

The other direction is immediate.  Assume to the contrary that for some $0<\rho <1,$ $f_\rho$ is not univalent. Then there exists a point $\omega_0$ such that the total change of the argument of  $f_\rho(z)-\omega_0 $ around $f_\rho (\partial \mathbb D_{\rho})$ is
$$\frac1{2\pi}\Lambda_{\partial \mathbb D_{\rho}}\arg (f_\rho(z)-\omega_0)\ge 2, $$
and by the inclusion property, the total change of the argument of $f(z)-\omega_0$ around $f (\partial \mathbb D)$ is greater or equal to 2.
Each $g_r, \;0<r<1$ is sense preserving and univalent  in $\mathbb D_r.$  A limiting process as $r$ tends to one shows that $g_1(z)$ is univalent in $\mathbb D.$ But $g_1(z)\equiv f(z)$ for all $z\in\partial \mathbb D$ and, hence,
$$ 1=\frac1{2\pi}\Lambda_{\partial \mathbb D}\arg (g(z)-\omega_0)=\frac1{2\pi}\Lambda_{\partial \mathbb D}\arg (f(z)-\omega_0)\geq 2,$$
which leads to a contradiction.
\end{proof}

\subsection*{Proof of Theorem \ref{main1}}
 Assume that there exists a univalent polyharmonic function $f$ from the unit disk onto $\mathbb C. $ Then  by Lemma \ref{lem1}, $g_1(z)$ is a univalent harmonic mapping that maps $\mathbb D$ onto $\mathbb  C.$ By Theorem A, such a mapping does not exist, and we get a contradiction.
 \qed

\subsection*{Proof of Theorem \ref{main2}}
By Lemma \ref{lem1}, we conclude that $g(z)= F_1(z)+F_2(z)$ is a univalent harmonic mapping in $\mathbb D.$ Let $g^*(e^{it})$ denote the boundary of $g(z),$ then by Theorem C, a jump at $g^*( e^{i\theta_0})$ occurs if, and only if,
$$\lim_{r\to 1^{-}}(1-r)\big|H'_1(re^{i\theta_0})+H'_2(re^{i\theta_0})\big|=c>0.$$
On the other hand
\begin{eqnarray}
g^*(e^{it})-f^*(e^{it})&=& \lim_{z \rightarrow e^{it}}\left ( F_1(z)+ F_2(z)\right )-\lim_{z \rightarrow e^{it}}\left ( F_1(z)+|z|^2 F_2(z)\right )\nonumber \\
 &=&\lim_{z \rightarrow e^{it}} \big(1+|z|)(1-|z|)F_2(z)\big )\nonumber \\
 &=&0,\nonumber
\end{eqnarray}
and the result follows.
\qed



\begin{rem}
Under the assumptions of Theorem \ref{main2} above, observe that if both $F_1$ and $F_2$ have continuous boundary functions at $e^{i\theta_0}$, then $f^*$ is also continuous at the same point. On the other hand, if one of $F^*_1$, $F^*_2$ is continuous and the other one has a jump at $e^{i\theta_0}$, then $f^*$ must have a jump.
\end{rem}

\subsection*{Proof of Theorem \ref{main4}}
Consider the harmonic mapping $g(z)=F_1(z)+F_2(z)$ in the unit disk, and denote by $g^*$ its boundary function. By Lemma \ref{lem1}, and the proof of Theorem \ref{main2}, we deduce that $g$ maps $\mathbb{D}$ univalently onto $\Omega$, with the dilatation $a_g=\phi$.

Hence, by \cite[Theorem 3]{BCER}, we have that if 
$$\int_{\Gamma _{m,\theta_0}} \frac{1-| a_g(z)|^2}{1-|z|^2}|dz|=\infty, \text{  for all  } 0<m<1/\pi,$$ 
then $g^*$ is continuous at the point $e^{i\theta}$. As the boundary functions of $f$ and $g$ coincide, the same holds for $f^*$.
\qed

To prove Theorem \ref{main3}, we shall need the following theorems due to Clunie and Sheil-Small in \cite{CS}.

\begin{Thm} \label{Thm5.7}
If $f=h+\overline{g} \in K_H$, then the functions $h+\varepsilon g$ are close-to-convex for $| \varepsilon | \leq 1$.
\end{Thm}

\begin{Thm}\label{Lem5.15}
Let $h$ and $g$ be analytic in $\mathbb{D}$ with $|g'(0)|<|h'(0)|$ and $h+\varepsilon g $ be close-to-convex for each
$\varepsilon$; $|\varepsilon|=1.$ Then $h+\overline g $ is harmonic close-to-convex function.
\end{Thm}

\subsection*{Proof of Theorem \ref{main3}}
Let the harmonic mapping $F=H+\overline{G}\in K_H$. Theorem~\Ref{Thm5.7} implies that $H+\varepsilon G$ is an analytic close-to-convex function in $\mathbb{D}$ for all $| \varepsilon | \leq 1$. Hence, $$H(z)+\varepsilon G(z)$$ maps conformally the unit disk onto a close-to-convex domain and, by the hereditary property of close-to-convex analytic mappings, it follows that
$$H(\rho z)+ \varepsilon G(\rho z)$$
 is close-to-convex for all $0<\rho <1$, $z\in \mathbb{D}_{1/\rho}$. Let $\varepsilon = \rho^2 \eta$ for $0<\rho <1$ and $| \eta |=1$. Then,
$$H(\rho z)+ \eta \rho^2 G(\rho z)$$
is close-to-convex in $\mathbb{D}$ for all $| \eta |=1$ and Theorem \Ref{Lem5.15} implies that
$$H(\rho z)+\rho^2 \overline{G(\rho z)}$$
is a univalent close-to-convex harmonic mapping in $\mathbb{D}$. Written otherwise,
$$g_\rho(z)=H(z)+\rho^2\overline{G(z)}$$
is a univalent close-to-convex harmonic mapping in $\mathbb D_\rho$.
By Lemma \ref{lem1} and the construction of $g_\rho$, we conclude that $f(z)=H(z)+|z|^2 \overline{G(z)}$ is a univalent fully close-to-convex biharmonic mapping in $\mathbb{D}$.
\qed

\section{Example}

In is section, we give a simple example illustrating applications of Theorem  \ref{main3}.

\begin{example}\label{ex3}
Consider the function $F(z)=z-\frac{\overline{z}^2}{6}$, defined in the unit disk.
Because
$$\sum_{n=2}^\infty n^2 |a_n|+ \sum_{n=1}^\infty n^2 |b_n| =2/3 <1,$$
where $a_n$, $b_n$ are the Taylor coefficients of the analytic and the anti-analytic parts of $F$ respectively,
by  \cite[Theorem 3]{JS}, it follows that $F\in K_H$. 

The dilatation of the biharmonic function 
$$f(z)=z-|z|^2\frac{\overline{z}^2}{6}$$
is 
\[
a_f(z)=-\frac{3z^2\overline{z}}{6-\overline{z}^3},
\]
which is bounded by one, and $f_z(z)\not =0$ in $\mathbb D $, so that 
$f(z)$ is locally univalent. Thus, by Theorem \ref{main3}, $f(z) $
is fully close-to-convex in the unit disk. See Figure \ref{fig1}.
\end{example}

\begin{figure}[ht]

\includegraphics[width=4.7cm]{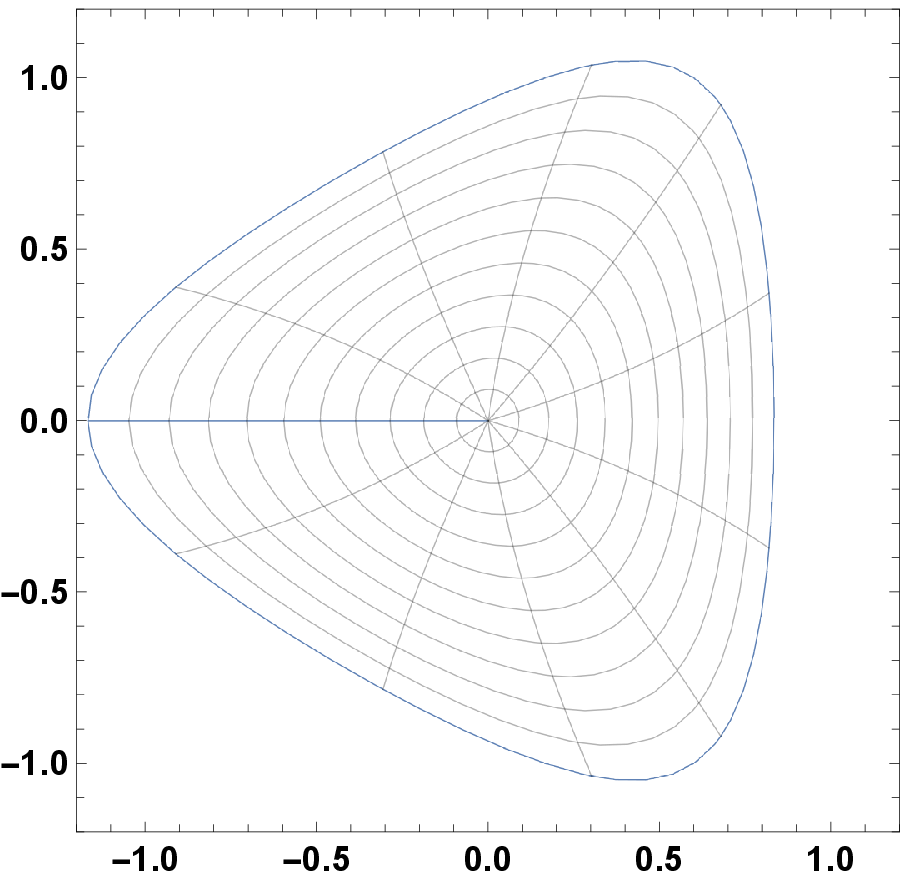}
\includegraphics[width=4.7cm]{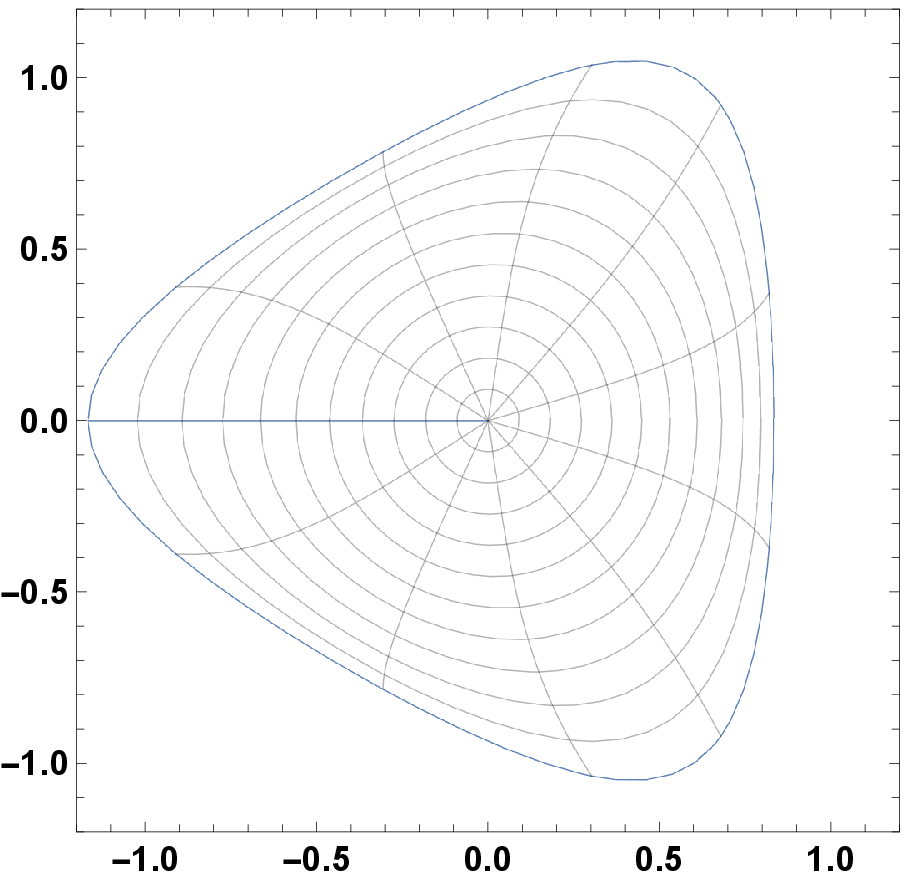}
\includegraphics[width=4.7cm]{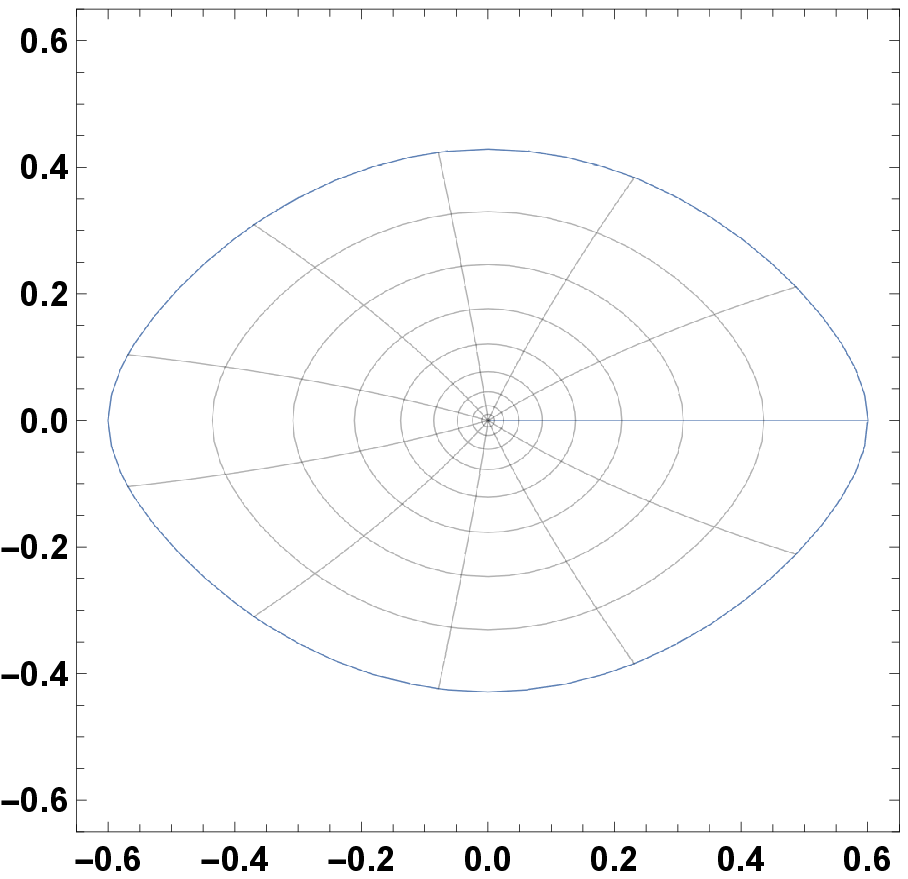}

\caption{Images of the unit disk under the functions $F$, $f$, and $a_f$ of Example \ref{ex3}.}\label{fig1}
\end{figure}

\end{document}